\documentclass{amsart}
\usepackage{fullpage, bbm}
\usepackage{amsfonts}
\usepackage{amssymb}
\usepackage{amsmath}
\usepackage{amsthm}
\usepackage{latexsym}

\newtheorem{theorem}{Theorem}[section]
\newtheorem{lemma}[theorem]{Lemma}

\newtheorem{proposition}[theorem]{Proposition}

\newtheorem{assumption}[theorem]{Assumption}

\theoremstyle{remark}\newtheorem{remark}{Remark}

\theoremstyle{remark}

\newcommand{\EE}{\mathbb{E} }
\newcommand{\one}{\mathbbm{1} }
\newcommand{\PP}{\mathbb{P} }
\newcommand{\RR}{\mathbb{R} }

\newcommand{\ff}{\mathcal{F} }

\newcommand{\HS}{L_{\mathrm{HS}} }

\begin{document}

\author{Si Cheng and Michael R. Tehranchi\\
University of Cambridge }
\address{Statistical Laboratory\\
Centre for Mathematical Sciences\\
Wilberforce Road\\
Cambridge CB3 0WB\\
UK}
\email{m.tehranchi@statslab.cam.ac.uk}

\title{Spectral term structure models}


\begin{abstract}
This note studies a certain stochastic evolution equation in the space of probability measures,
including existence and uniqueness results. A solution of this equation gives rise, in a natural way,
to an interest rate term structure model, in the same spirit as the Heath-Jarrow-Morton framework.
\end{abstract}

\maketitle

In this note, we are interested in a stochastic process $(\mu_t)_{t \ge 0}$ taking values in the space of (Borel)
probability  measures on $\RR$ 
and
 whose stochastic evolution can be described formally by the equation
\begin{equation}\label{eq:evol}
d\left[\mu_t( dr) \right] = (R_t- r) \mu_t(dr) dt +  M(dr \times dt)
\end{equation}
where
$$
R_t = \int_{\RR} r \ \mu_t(dr)
$$
and $M$ is a random signed measure on $\RR \times \RR_+$ with the
property that $M( \RR \times (0,t] ) = 0$ for all $t \ge 0$.
We will give a more rigorous account of the evolution equation \eqref{eq:evol} later, but for the moment,
one should interpret it to mean  that the real-valued process $M^\varphi$ defined by
$$
M_t^\varphi = \int_{\RR} \varphi(r) \mu_t(dr) - \int_{\RR} \varphi(r) \mu_0(dr)  -  \int_0^t \left(\int_{\RR}
 \varphi(r)(R_s-r) \mu_s(dr) \right) ds
$$
is a  local  martingale for all test functions $\varphi:\RR \to \RR$ in some suitable collection.

The reason for our interest in the process $(\mu_t)_{t \ge 0}$ is contained in the following
computation.
For $0 \le t \le T$, let
$$
P(t,T) =  \int_{\RR} e^{-(T-t) r} \mu_t(dr).
$$
Then  by a formal application of It\^o's formula we have
$$
e^{- \int_0^t R_u   du } P(t,T) = P(0,T) + \int_{\RR \times (0,t]}  e^{- \int_0^s R_u   du  -(T-s) r} M(dr \times ds).
$$
Of course, the  stochastic integral on the right-hand side must be interpreted carefully.
  Nevertheless, if we proceed
optimistically, we can hope to find suitable assumptions such that the right-hand side
is a true martingale, and, in particular, since $P(T,T) = 1$, we have will have
$$
\EE\left[ e^{- \int_t^T R_s ds } | \ff_t \right]= P(t,T)
$$
for all $0 \le t \le T.$  The above formula has a financial interpretation.
Consider a continuous-time market model where the
the time-$t$ spot interest rate is $R_t$ and  the time-$t$ price of the zero-coupon bond of maturity $T$
is $P(t,T)$.
Then the underlying probability measure $\PP$ is a risk-neutral measure for the model,
and in particular, the bond market has no arbitrage.

In term structure modelling, it is often desirable to have non-negative interest rates.
One appealing feature of this framework is that
 it is very easy to ensure that the interest rate $r_t$ is non-negative:
it is  sufficient  that the measure $\mu_t$ is
supported on $[0,\infty)$.

The above form of modelling the interest rate term structure
is inspired by the recent paper \cite{S} of Siegel.
Siegel's modelling scheme can be described as follows.  Fix $n \ge 2$, and $n$ real numbers
$r_1, \ldots, r_n$. Suppose the processes
$X^1, \ldots, X^n$ evolve as
$$
dX_t^i = X_t^i (R_t - r_i) dt + dM^i_t
$$
for all $1 \le i \le n$, where
$$
R_t = \sum_{i=1}^n r_i X^i_t
$$
and the processes $M^1, \ldots, M^n$ are local martingales.
The measure $\mu_t$ defined by
$$
\mu_t = \sum_{i=1}^n X_t^i \delta_{r_i},
$$
satisfies the evolution equation \eqref{eq:evol},
where the notation $\delta_r$ denotes the Dirac measure concentrated at $r$
and formally
$$
M(dr \times dt) = \sum_{i=1}^n    \delta_{r_i}(dr)  dM_t^i.
$$

The goal of this paper is to treat the infinite dimensional version of the Siegel's
model in the spirit of Filipovic's account \cite{filipovic} of the Heath--Jarrow--Morton
term structure framework.
Two technical challenges appear  in studying the formal evolution equation \eqref{eq:evol}.
The first challenge is to deal with the nonlinearity appearing in the drift.  Indeed,
the nonlinearity is quadratic, and thus not globally Lipschitz.
Our solution to this problem is to restrict our attention to measures with bounded support, in which
case it is possible to treat the nonlinearity as though it were Lipschitz.
The second challenge is
how to define properly the stochastic integral with respect to a local martingale taking
values in a possibly infinite-dimensional space of signed measure. We by-pass this
difficulty by letting our local martingale take values in a larger space of
distributions which can be endowed with the structure of a separable Hilbert space, and
then appealing to the well-known Hilbert space stochastic integration theory.
Finally, to ensure that the distribution valued process $(\mu_t)_{t \ge 0}$   actually takes values in the set
of probability measures, we employ a discretisation argument.

\section{Set-up and mathematical preliminaries}
Although the object of interests $\mu_t$ are probability measures,
we study the evolution equation \eqref{eq:evol} in the well-known Hilbert space framework,
in the style of the book of Da Prato \& Zabzcyck \cite{daprato-zabzcyck}.

\subsection{The state space}
Fix a bounded closed interval $I \subset \RR$ containing the origin.
For an absolutely continuous function $\varphi$ on $I$ let
$$
\| \varphi \|_H =   \varphi(0)^2 + \int_I  \varphi'(r)^2 dr
$$
Let $H$ be the space of absolutely continuous functions $\phi$ such that
$\| \phi \|_H < \infty$.  It is well-known that $H$ is a separable Hilbert space
with respect to the norm $\| \cdot \|_H$.

Now for a finite signed measure $\mu$ on $I$, let
$$
\| \mu \|_{H^*} = \sup_{ \varphi \in H, \| \phi \|_H = 1 } \int_I \varphi(r) \mu(dr)
$$
Let $H^*$ be the completion of the space of finite signed measures with respect to
this norm $\| \cdot \|_{H^*}$.   The space $H^*$ is the dual space of $H$ with respect
to the Banach structure of $H$.  We will occasionally use the notation
$\langle  \cdot , \cdot  \rangle_{H^*, H} : H^* \times H \to \RR$ to denote the duality pairing, so
that when $\mu$ is a signed measure
$$
\langle  \mu, \varphi  \rangle_{H^*, H } =  \int_I \varphi(r) \mu(dr).
$$

But since $H$ is a separable Hilbert space, we
know by the Riesz representation theorem that the dual space $H^*$
can be identified isometrically with $H$.
Indeed, assuming $I = [-a,b]$, each signed measure $\mu \in H^*$ corresponds to a function $\varphi \in H$
by
$$
\varphi(r) = \mu(I) + \left\{\begin{array}{ll} \int_0^r \mu(s, b] ds & \mbox{ if } r> 0 \\ \int_{r}^0 \mu[-a,s) ds & \mbox{ if }
r \le 0 \end{array} \right.
$$
and hence the norm can be computed for signed measures $\mu$ by the formula
$$
\| \mu \|_{H^*}^2 = | \mu(I) |^2 + \int_{-a}^0 \mu[-a,r)^2 dr + \int_0^b \mu( r, b]^2 dr.
$$

Finally, we single out an important subset:
$$
\mathcal P = \{ \mbox{ probability measures on } I \} \subset H^*.
$$

\subsection{The local martingale}
We now explain how we construct the `random signed measure'   appearing in
equation \eqref{eq:evol}.   Given the set-up described above,
we will consider local martingales valued in $H^*$.  Such local martingales
will be built from stochastic integrals with respect to a cylindrical
Brownian motion.

Let $G$ an arbitrary real separable Hilbert.  Since the specific structure
of $G$ is irrelevant, we will identify the dual space $G^*$ with $G$ without comment
and denote the inner product by $\langle \cdot, \cdot \rangle_G : G \times G \to \RR$.
Indeed, the reader may let $G = \ell^2$ without loss.

Let $W$ be a Brownian motion defined cylindrically in $G$. Recall that this
means that
$$
W= \{ W_t(g):  t \ge 0, g \in G \}
$$
is such that for each $g \in G$ such that $\|g \|_G = 1$, the process $W(g)$ is a
Brownian motion, and $W$ is linear in the sense that $W(\alpha g+ \beta h) = \alpha W(g) + \beta W(h)$
for all $g, h \in G$ and $\alpha, \beta \in \RR$.
Heuristically, the cylindrical Brownian motion can be realised by the formal
sum
$$
W_t = \sum_{i=1}^{\infty} W_t^i e_i.
$$
where  $W^1, W^2, \ldots$ are independent Brownian motions  and $(e_n)_n$
is a complete orthonormal basis of $G$.

For a linear map  $A: G \to H^*$ define
$$
\| A \|_{\HS}^2  = \sum_n \| A e_n \|_{H^*}^2,
$$
where $(e_n)_n$ is a complete orthonormal basis of $G$ as before.
In fact,   the norm $\| \cdot \|_{\HS}$ is independent
of the choice of basis $(e_n)_n$.
The  space of operators $A$ such that
$\| A \|_{\HS} < \infty$ are the Hilbert--Schmidt
operators $\HS(G,H^*)$.  Occasionally we will identify
an operator in $\HS(G,H^*)$ with a vector in the tensor
product space $H^* \otimes G$ in the obvious way. Finally, for
an element $A \in \HS(G,H^*)$ we let $A^* \in \HS(H, G)$ be
the adjoint defined by
$$
\langle Ag, \varphi \rangle_{H^*, H} = \langle g, A^* \varphi \rangle_G \mbox{ for all } g \in G, \varphi \in H.
$$

We recall basic result of the Hilbert space integration theory used here.
\begin{proposition}
Let $(\sigma_t)_{t \ge 0}$ be a predictable process taking values in $\HS(G, H^*)$
and such that
$$
\int_0^t \| \sigma_s \|^2_{\HS}ds < \infty \mbox{ almost surely for all } t \ge 0.
$$
Then the stochastic integral
$$
M_t = \int_0^t \sigma_s dW_s
$$
is well-defined, and the process $M$ is a continuous local martingale taking values in $H^*$.
\end{proposition}

For a proof see \cite{CT} or \cite{daprato-zabzcyck}.

 \subsection{An aside on technical conventions}
Because an infinite dimensional Hilbert space such as $H^*$ or $\HS(G, H^*)$ can
be equipped with several inequivalent topologies, it is necessary to make some conventions
to clarify our meaning.
When we speak
of measurable maps into a Banach space, we mean measurable with respect
to the Borel sigma-field generated by its norm topology. In fact, this notion of measurability is equivalent to the a priori weaker notion of weak measurability if the Hilbert space is separable \cite{RS}.Also,
when we say that $M$ is a martingale   valued in a Hilbert space, we mean that the
real random variable $\| M_t \|$
is integrable and the conditional expectation
$$
\EE( M_t | \ff_s) = M_s \mbox{ for all } 0 \le s \le t
$$
 is interpreted in the sense of   Bochner.
 Finally, when we say a Hilbert space-valued process   is continuous,
we mean the almost sure continuity of the sample paths with respect to the norm topology.

\section{Existence and uniqueness}\label{se:result}

We now consider the evolution equation \eqref{eq:evol}
with extra structure that
$$
M_t = \int_0^t \sigma( \mu_s) dW_s.
$$
To make things precise, we study the stochastic differential equation
\begin{equation}\label{eq:evol1}
d\mu_t = (R(\mu_t)-\rho^*) \mu_t dt + \sigma( \mu_t) dW_t
\end{equation}
where  $W$ is a cylindrical Brownian motion in the Hilbert space $G$;
the map $\sigma: \mathcal P \to \HS(G, H^*)$
is given;
the bounded linear operator $\rho : H \to H$ is defined by
$$
(\rho \varphi)(r) = r \varphi(r) \mbox{ for all } \varphi \in H, r \in I,
$$
and its adjoint $\rho^*: H^* \to H^*$ is defined by
$$
\langle \rho^* \mu, \varphi \rangle_{H^*,H} = \langle  \mu, \rho \varphi \rangle_{H^*,H} \mbox{ for all } \mu \in H^*, \varphi \in H,
$$
and the linear map $R: H^* \to \RR$ is
defined by
$$
R(\mu) = \langle \rho^* \mu, \mathbf{1} \rangle_{H^*, H}
$$
where $\mathbf{1} \in H$ is defined by
$$
\mathbf{1}(r) = 1 \mbox{ for all } r \in I.
$$

Note that if $\mu \in H^*$ is a signed measure then
$$
(\rho^* \mu)(dr) = r \mu(dr)
$$
and
$$
R(\mu) = \int_I r  \ \mu(dr),
$$
so the SDE given in equation \eqref{eq:evol1} captures the main features of our evolution equation \eqref{eq:evol}.

We now make the following additional assumptions:
\begin{assumption}\label{as:1} The map $\sigma: \mathcal P \to \HS(G, H^*)$ has the following properties:
\begin{description}
\item[Centering]
For all $\mu \in \mathcal P$, we have
$$
  \sigma(\mu)^* \mathbf{1}  =  0.
$$
\item[Lipschitz]
There is a constant $C > 0$ such that for all $\mu, \nu \in \mathcal P$ we have
$$
\| \sigma(\mu) - \sigma(\nu) \|_{\HS} \le C \| \mu - \nu \|_{H^*}.
$$
\item[Absolute continuity]
There exist a function  $g : \mathcal P \times I \to G$ and a constant $C > 0$ such that
$\| g(\mu, r) \|_G \le C$ for all $\mu \in \mathcal P, r \in I$ and
$$
\sigma(\mu)(dr) = g( \mu, r) \mu(dr).
$$
\end{description}
\end{assumption}

With this preparation, we now can state the main result:
\begin{theorem}\label{th:solex}
For every $\mu_0 \in \mathcal P$, the SDE given by equation \eqref{eq:evol1} has a unique
solution $(\mu_t)_{t \ge 0}$ valued in the space of probability measures $\mathcal P$ on $I$.
\end{theorem}

\section{Proofs}\label{se:proof}
\noindent
\textit{Step 1: Estimates of integrals}

In this subsection, we are going to prove:
\begin{theorem}\label{th:solesti}
Suppose $\bar{\mu} = (\mu_t)_{t \geq 0}$ and $\bar{\nu} = (\nu_t)_{t \geq 0}$ are two $\mathcal P$-valued solutions of the SDE \eqref{eq:evol1}. Then $\bar{\mu}, \bar{\nu} \in \mathcal{S}_T$ for any $T > 0$. And there is a constant $K > 0$ such that for all $T \ge 0$ we have
\begin{equation}\label{eq:esti3}
\EE( \sup_{0 \le t \le T} \| \mu_t- \nu_t\|_{H^*}^2 ) \leq \| \mu_0- \nu_0\|_{H^*}^2 e^{KT^2}
\end{equation}
In particular, if SDE \eqref{eq:evol1} has solution, it must be unique.
\end{theorem}

Let $(W_t)_{t \geq 0}$ be a Brownian motion defined cylindrically on $G$ with filtration $(\ff_t)_{t \geq 0}$ defined on some background probability space $(\Omega, \ff, \PP)$. Fix any $T \geq 0$ and define $\mathcal{S}_T$ to be the set of continuous $H^{*}$-valued processes $\bar{\mu} = (\mu_t)_{0 \leq t \leq T}$, adapted to filtration $(\ff_t)_{t \geq 0}$, that has finite $|\|\cdot\||_T$ norm. Where the norm on $\mathcal{S}_T$ is defined as
$$
|\| \bar{\mu} \| |_T := \big{\|}\sup_{t \leq T} \| \mu_t \| \big{\|}_2
$$
Where $\| X \|_2^2 = \EE(X^2)$ is the $L^2$-norm.
\begin{remark}
The continuous property of $\bar{\mu}$ allows us to make sense of Lebesgue integrals of $\bar{\mu}$ by defining:
$$
<\phi, \int_0^t \mu_s ds> := \int_0^t <\phi,\mu_s> ds
$$
since for any test function $\phi \in H$, the function $t \mapsto <\phi, \mu_t(\omega)>$ is bounded continuous and hence integrable.
\begin{align*}
| <\phi, \mu_t> | &\leq \| \phi \| \|\mu_t \| \leq \|\phi \| (\sup_{t \leq T} \| \mu_t \|) < \infty \quad \text{a.s.} \\
|<\phi, \mu_t> - <\phi, \mu_s>| &= |<\phi, \mu_t - \mu_s>| \leq \| \phi \| \| \mu_t - \mu_s \| \rightarrow 0 \quad \text{as} \quad t \rightarrow s
\end{align*}
\end{remark}
\begin{lemma}
For any $T \geq 0$, the metric space $(\mathcal{S}_T, |\| \cdot \||)$ is complete.
\end{lemma}
\begin{proof}
Let $(\bar{\mu}^n)_n$ be Cauchy in $\mathcal{S}_T$. i.e. $|\| \bar{\mu}^n - \bar{\mu}^m \|| \rightarrow 0$ as $m,n \rightarrow 0$. We can find a subsequence $(n_k)_k$ such that
$$
\sum_k |\| \bar{\mu}^{n_{k+1}} - \bar{\mu}^{n_k} \|| < \infty
$$
By applying triangle inequality of $L^2$ norm, we have
$$
\| \sum_k \sup_{t\leq T} \| \mu_t^{n_{k+1}} - \mu_t^{n_k} \|_2 \leq \sum_k |\| \bar{\mu}^{n_{k+1}} - \bar{\mu}^{n_k} \|| < \infty
$$
Hence for almost every $\omega \in \Omega$,
$$
\sum_k \sup_{t\leq T} \| \mu_t^{n_{k+1}}(\omega) - \mu_t^{n_k}(\omega) \| < \infty
$$
Especially $(\mu_t^{n_k}(\omega))_k$ is Cauchy in $H^{*}$ for every $t$ and hence converges to some $\mu_t(\omega)$ by completeness of $H^{*}$. Fix any $\epsilon > 0$, we can also find some $j(\omega)$ such that
$$
\sum_{k=j(\omega)}^\infty \sup_{t\leq T} \| \mu_t^{n_{k+1}}(\omega) - \mu_t^{n_k}(\omega) \| < \tfrac{\epsilon}{2}
$$
Now given any $t \leq T$, since $\mu_t^{n_k}(\omega) \rightarrow \mu_t(\omega)$, there exist some constant $K(t,\omega)$ such that $\| \mu_t^{n_k}(\omega) - \mu_t(\omega) \| < \tfrac{\epsilon}{2}$, whenever $k > K(t,\omega)$. Then we have the following estimate for any $t \leq T$:
\begin{align*}
\| \mu_t^{n_{j(\omega)}}(\omega) - \mu_t(\omega) \| &\leq \sum_{k=j}^{K(t,\omega)} \| \mu_t^{n_{k+1}}(\omega) - \mu_t^{n_k}(\omega) \| + \| \mu_t^{n_{K(t,\omega)}+1}(\omega) - \mu_t(\omega) \| \\
                                          &< \tfrac{\epsilon}{2} + \tfrac{\epsilon}{2} < \epsilon
\end{align*}
In other word, we've found a constant $n_{j(\omega)}$ such that
$$
\sup_{t \leq T} \| \mu_t^{n_{j(\omega)}}(\omega) - \mu_t(\omega) \| < \epsilon
$$
Hence
\begin{align*}
\| \mu_t(\omega) - \mu_s(\omega) \| &\leq  \| \mu_t(\omega) - \mu_t^{n_{j(\omega)}}(\omega) \| + \| \mu_t^{n_{j(\omega)}}(\omega) - \mu_s^{n_{j(\omega)}}(\omega) \| + \| \mu_s^{n_{j(\omega)}}(\omega) - \mu_s(\omega) \| \\
                                    &\leq 2 \sup_{t \leq T} \| \mu_t^{n_{j(\omega)}}(\omega) - \mu_t(\omega) \| + \| \mu_t^{n_{j(\omega)}}(\omega) - \mu_s^{n_{j(\omega)}}(\omega) \| \\
                                    &< 3\epsilon \quad \text{if $s,t$ are close enough}
\end{align*}
Hence $\bar{\mu} = (\mu_t)_{t \leq T}$ is a.s continuous. i.e. $\bar{\mu} \in \mathcal{S}_T$ and it remains to show that this is indeed the limit of the original sequance $(\bar{\mu}^n)_n$, but
\begin{align*}
|\| \bar{\mu}^n - \bar{\mu} \||^2 &= \EE \left( \sup_{t \leq T} \| \mu_t^n - \mu_t \|^2 \right) \\
                                  &= \EE \left( \sup_{t \leq T} \liminf_{k \rightarrow \infty}  \| \mu_t^n - \mu_t^{n_k} \|^2 \right) \\
                                  &\leq \liminf_{k \rightarrow \infty} \EE \left( \sup_{t \leq T} \| \mu_t^n - \mu_t^{n_k} \|^2 \right) \quad \text{Fatou's lemma} \\
                                  &= \liminf_{k \rightarrow \infty} |\| \bar{\mu}^n - \bar{\mu}^{n_k} \||^2 \quad \rightarrow 0 \quad \text{as } n \rightarrow \infty
\end{align*}
because $(\bar{\mu}^n)_n$ is a Cauchy sequence in $\mathcal{S}_T$.
\end{proof}

Suppose now $\bar{\mu} = (\mu_t)_{t \geq 0}$ and $\bar{\nu} = (\nu_t)_{t \geq 0}$ are two solutions to SDE \eqref{eq:evol1} such that $\mu_t, \nu_t \in \mathcal{P}$ for all $t$. Then we have the following estimate:
\begin{align*}
\| \mu_t - \nu_t \|^2 &= \| \mu_0 -\nu_0 + \int_0^t F(\mu_s) - F(\nu_s) ds + \int_0^t \sigma(\mu_s) - \sigma(\nu_s) dW_s \|^2 \\
                      &\leq 3 \left(\|\mu_0 -\nu_0\|^2 + \| \int_0^t F(\mu_s) - F(\nu_s) ds \|^2 + \| \int_0^t \sigma(\mu_s) - \sigma(\nu_s) dW_s \|^2 \right)
\end{align*}
Where $F(\mu):= \left( R(\mu) -\rho^* \right) \mu$. Hence we have:
\begin{equation}\label{eq:esti1}
|\| \bar{\mu} - \bar{\nu} \||_T^2 \leq 3\left(\| \mu_0 - \nu_0 \|^2 + |\| \int_0^{\cdot} F(\mu_s) - F(\nu_s) ds \||_T^2 + |\| \int_0^{\cdot} \sigma(\mu_s) - \sigma(\nu_s) dW_s \||_T^2 \right)
\end{equation}
\vspace{2ex}

First we estimate the final stochastic term of \eqref{eq:esti1}. Notice that if $\bar{\mu} \in \mathcal{S}_T$ for some $T$, then $M_t = \int_0^t \sigma(\mu_s) dW_s$ is a local martingale. The expectation of quadratic variation is then
\begin{align*}
\EE [M]_T &= \EE \int_0^T \| \sigma(\mu_s) \|^2 ds \quad \text{Ito's isometry} \\
          &\leq 2 \EE \int_0^T \| \sigma(0) \|^2 + C^2 \|\mu_s\|^2 ds \quad \text{Lipschitz property}\\
          &\leq 2T(\| \sigma(0) \|^2 + C^2 |\|\bar{\mu} \||_T^2) < \infty
\end{align*}
Hence $(M_t)_{t \leq T}$ is a true martingale and bounded in $L^2$.
\begin{theorem}[Doob's $L^2$ inequality]\label{th:Doobs}
Let $H$ be a separable Hilbert space and $M = (M_t)_{t \geq 0}$ be a $H$-valued martingale, then
$$
\EE \| \sup_{t \leq T} M_t \|^2_H \leq 4 \EE \|M_T\|^2_H
$$
\end{theorem}
By applying above theorem, we have
\begin{align*}
|\| \int_0^{\cdot} \sigma(\mu_s) - \sigma(\nu_s) dW_s \||_T^2 &= \EE \left( \sup_{t \leq T}  \int_0^t \| \sigma(\mu_s) - \sigma(\nu_s) \|^2 ds\right) \\
                                                              &\leq 4\EE \int_0^T \| \sigma(\mu_s) - \sigma(\nu_s) \|^2 ds \\
                                                              &\leq 4C^2 \int_0^T \EE \| \mu_s -\nu_s \|^2 ds \\
                                                              &\leq 4C^2 \int_0^T |\| \bar{\mu} - \bar{\nu} \||_s^2 ds \\
\end{align*}

Now we estimate the second term of \eqref{eq:esti1}. Notice that if $\bar{m} = (m_s)_{s \leq T} \in \mathcal{S}_T$, then
\begin{align*}
\| \int_0^t m_s ds \| &= \sup_{\substack{
   h \in H \\
   \| h \|_H =1
  }} | < h, \int_0^t m_s ds>| \\
                      &\leq \sup_{\substack{
   h \in H \\
   \| h \|_H =1
  }} \int_0^t |<h,m_s>| ds \\
                      &\leq \int_0^t \|m_s \| ds
\end{align*}
\begin{lemma}
For any $\mu, \nu \in H^{*}$,
$$
\| F(\mu) - F(\nu) \| \leq \sqrt{2(a+b)} \|\mu + \nu \| \|\mu - \nu \| + \sqrt{2((a+b)^2+2(a+b)+2)} \|\mu - \nu \|
$$
In particular, if $\mu, \nu \in \mathcal{P}$ are probability measures on interval $I = [-a,b]$, we have
$$
\| F(\mu) - F(\nu) \| \leq C_1 \|\mu - \nu \|
$$
where $C_1 = 2\sqrt{2(a+b)(1+a+b)} + \sqrt{2((a+b)^2+2(a+b)+2)}$.
\end{lemma}
\begin{proof}
By previous notation $R(\mu) = \int_I r \mu(dr)$ is linear in $\mu$ and
$$
|R(\mu)| = |<r,\mu>| \leq \|r\|_H \|\mu\|_{H^{*}} = \sqrt{a+b} \|\mu\|
$$
Now define $(\mu \ast \nu)(dr) := R(\mu) \nu (dr)$, a bilinear binary operation.
\begin{align*}
\| \mu \ast \nu\| &= \sup_{\|h\| = 1} |< h, \mu \ast \nu>| \\
                  &= \sup_{\|h\| = 1} |R(\mu)| |<h,\nu>| \\
                  &\leq \sqrt{a+b} \|\mu\| \|\nu\|
\end{align*}
For any $h(r) \in H$, we have $(rh(r))' = h(r) + rh'(r)$ and
\begin{align*}
\|rh(r)\|^2_H &= \int_{-a}^b \left( h(r) + rh'(r) \right)^2 dr \\
              &\leq 2\left( \int_{-a}^b h^2(r) dr + \int_{-a}^b r^2 (h'(r))^2 dr \right)
\end{align*}
The second term can be bounded easily by $(a+b)^2 \|h\|^2$, for the first term, we have:
\begin{align*}
h(r) &= h(0) + \int_0^r h'(s)ds \\
h^2(r) &\leq 2 \left( h^2(0) + \left( \int_0^r h'(s) ds \right)^2 \right) \\
       &\leq 2 \left( h^2(0) + r \int_0^r (h'(s))^2 ds \right) \quad \text{Cauchy-Schwarz inequality} \\
       &\leq 2(a+b+1)\|h\|^2
\end{align*}
Combing the two terms gives:
$$
\| rh(r)\|^2 \leq 2((a+b)^2+2(a+b)+2)\|h\|^2
$$
Hence for any $\mu \in H^{*}$
\begin{align*}
\|\rho^* \mu\| &= \sup_{\|h\| = 1} |<h(r), \rho^* \mu>| \\
         &= \sup_{\|h\| = 1} |<rh(r), \mu>| \\
         &\leq \sup_{\|h\| = 1} \|rh(r)\| \|\mu\| \\
         &\leq \sqrt{2((a+b)^2+2(a+b)+2)} \|\mu\|
\end{align*}
Finally grouping all the above estimates gives:
\begin{align*}
\|F(\mu) - F(\nu) \| &= \| \mu \ast \mu - \nu \ast \nu + \rho^* (\mu - \nu) \| \\
                     &\leq \Big{\|} \frac{(\mu + \nu) \ast (\mu - \nu) + \mu - \nu) \ast (\mu + \nu)}{2} \Big{\|} + \|\rho^* (\mu - \nu)\| \\
                     &\leq \sqrt{2(a+b)} \|\mu + \nu \| \|\mu - \nu \| + \sqrt{2((a+b)^2+2(a+b)+2)} \|\mu - \nu \|
\end{align*}
Notice that $\|\mu + \nu\| \leq \|\mu\| + \|\nu\|$ and for probability measures we have:
$$
||\mu||^2_{H^{*}} := \mu(I)^2 + \int_0^b \mu(r,N]^2 dr + \int_{-a}^0 \mu[-N,r)^2 dr \leq 1+a+b
$$
Which gives the final assertion of the lemma.
\end{proof}

Back to the estimate of the second term of \eqref{eq:esti1}:
\begin{align*}
|\| \int_0^{\cdot} F(\mu_s) - F(\nu_s) ds \||_T^2 &= \EE \sup_{t \leq T} \| \int_0^t F(\mu_s) - F(\nu_s) ds \|^2 \\
                                                  &\leq \EE \sup_{t \leq T} \left( \int_0^t \|F(\mu_s) - F(\nu_s)\| ds \right)^2 \\
                                                  &\leq \EE T \int_0^T \|F(\mu_s) - F(\nu_s)\|^2 ds \\
                                                  &\leq C_1^2 T \int_0^T |\| \bar{\mu} - \bar{\nu} \||_s^2 ds
\end{align*}
Hence the estimate \eqref{eq:esti1} takes the final form
\begin{equation}\label{eq:esti2}
|\| \bar{\mu} - \bar{\nu} \||_T^2 \leq 3\left(\| \mu_0 - \nu_0 \|^2 + C_2 \int_0^T |\| \bar{\mu} - \bar{\nu} \||_s^2 ds \right)
\end{equation}
Where $C_2 = 4C^2 + C_1^2 T$. Here we can use Gronwall's lemma, which states:
\begin{lemma}[Gronwall's lemma]
Let $T > 0$ and let $f$ be a non-negative bounded measurable function on $[0,T]$. Suppose that for some $\alpha,\beta \geq 0$:
$$
f(t) \leq \alpha + \beta \int_0^t f(s) ds \quad 0 \leq t \leq T.
$$
Then $f(t) \leq \alpha e^{\beta t}$ for all $t \in [0,T]$.
\end{lemma}
Hence by setting $f(t) = |\| \bar{\mu} - \bar{\nu} \||_t^2$, we get
$$
|\| \bar{\mu} - \bar{\nu} \||_T^2 \leq 3\| \mu_0 - \nu_0 \|^2 \exp(3C_2 T) = 3\| \mu_0 - \nu_0 \|^2 \exp(12C^2T + 3C_1^2 T^2)
$$
and hence we've showed the inequality in theorem \ref{th:solesti}.

It remains to verify the uniqueness assertion. If $\bar{\mu} = (\mu_t)_{t \geq 0}$ and $\bar{\nu} = (\nu_t)_{t \geq 0}$ are two solutions to \eqref{eq:evol1} with the same initial condition $\mu_0$. Then by Assumption \ref{as:1}, they must be $\mathcal P$-valued. Also since $\bar{\mu}, \bar{\nu}$ start at the same initial point,
$$
\EE( \sup_{0 \le t \le T} \| \mu_t- \nu_t\|_{H^*}^2 ) = 0 \quad \forall T > 0
$$
and hence $\mu_t = \nu_t$ for all $t \geq 0$

\vspace{3ex}
\noindent
\textit{Step 2: Existence and uniqueness in the atomic probability measure case:}

We call a (signed) measure $\mu$ atomic if it takes the form:
$$
\mu = \sum_{i=1}^k X^i \delta_{r_i}
$$
where $X_i, r_i \in \RR$ and $\delta_{r_i}$ is the dirac-delta measure concentrated at $r_i \in \RR$.
If in addition, $X^i > 0$ and $\sum_{i=1}^k X^i = 1$. Then $\mu$ is an atomic probability measure.

In this subsection, we are going to show
\begin{theorem}\label{th:ATOMEX}
Under Assumption \ref{as:1}, given any initial atomic probability measure $\mu_0$, there exists a solution $(\mu_t)_{t \geq 0}$ to the SDE \eqref{eq:evol1}
\end{theorem}

\begin{proposition}\label{th:atomprop}
Fix $n \ge 1$, and let
$$
\mathcal Q = \left\{ x= (x_1, \ldots, x_n): x_i \ge 0 \mbox{ for all } i, \sum_{i=1}^n x_i = 1 \right\}
$$
For $1 \le i \le n$, let $b_i: \mathcal Q \to \RR$ and $\sigma_i : \mathcal Q \to G$ be Lipschitz functions such that
$$
\sum_{i=1}^n b_i(x) = 0 \mbox{ and } \sum_{i=1}^n \sigma_i(x) = 0
$$
for all $x \in \mathcal Q$ and such there exists a constant $C  > 0$ such that
$$
| b_i(x) | + \|\sigma_i(x) \| \le C  x_i
$$
for all $x \in \mathcal Q$ where $x=(x_1, \ldots, x_n)$.

Then for every $\xi \in \mathcal Q$ there exists
unique adapted process $(X_t)_{t \ge 0}$ taking values
in $\mathcal Q$ such that $X_0= \xi$ and
$$
dX_t = b(X_t) dt + \sigma(X_t) dW_t.
$$
\end{proposition}

\begin{remark} We are interested in the following
situation.
  Fix a collection of real numbers
$r_1, \ldots, r_n$ and let
$$
b_i(x)  = x_i \left( \sum_{j=1}^n r_j x_j - r_i \right)
$$
for $x=(x_1, \ldots, x_n) \in \mathcal{Q}$.
\end{remark}

\begin{proof}
Uniqueness follows from the Lipschitz assumption in the usual way.  We need only prove existence.

Let $\Pi$ be the projection onto the closed convex set $\mathcal Q$.
Note that $\Pi$ is Lipschitz, and hence the functions $b \circ \Pi$ and $\sigma \circ \Pi$
are also Lipschitz.   Given $\xi \in \mathcal Q$, let
$(X_t)_{t \ge 0}$ be the unique strong solution to the SDE
$$
dX_t = b \circ \Pi(X_t) dt + \sigma\circ \Pi(X_t) dW_t.
$$
which exists by It\^o's theorem.   Note that by summing over the indices,  we have
$$
d\sum_{i=1}^n X^i_t = 0 \Rightarrow \sum_{i=1}^n X^i_t = 1 \mbox{ for all } t \ge 0.
$$
We need only show that $X_t^i \ge 0$  for all $t \ge 0$ and all $i$.  We would be
done since $\Pi(x)=x$ when $x \in \mathcal Q$.

First we suppose that $\xi_i > 0$ for all $1 \le i \le n$.
Let $T = \inf\{ t \ge 0: \min \hat X^i_t = 0 \}$.  We will now show that $T= \infty$ almost surely.
For each $1 \le i \le n$, we define bounded functions $c_i$ and $\tau_i$ by the formula
$$
c_i (x) =  \left\{ \begin{array}{ll} \frac{b_i(x)}{x_i} & \mbox{ if } x_i > 0 \\ 0 & \mbox{ otherwise } \end{array} \right.
$$
and
$$
\tau_i (x) =  \left\{ \begin{array}{ll} \frac{ \sigma_i(x)}{x_i} & \mbox{ if } x_i > 0 \\ 0 & \mbox{ otherwise } \end{array} \right.
$$
where $x=(x_1, \ldots, x_n) \in \mathcal{P}$.  In particular
$$
X_{t \wedge T}^i = \xi_i + \int_0^t X^i_{s \wedge T} dZ_t^i
$$
where $Z^i$ is the continuous semimartingale
$$
Z_t^i = \int_0^t \one_{\{ s \le T\}} [ c_i(X_s) ds + \tau_i(X_s) dW_s]
$$
Hence we can write
$$
X_{t \wedge T}^i = \xi_i e^{ Z^i_t - [Z^i]_t/2}.
$$
Since the right-hand side is strictly positive almost surely for all finite $t$, the
event $\{T < \infty\}$ must have probability zero, as claimed.

Now consider the case where there is at least one $i$ such that $\xi_i = 0$.  By relabelling
if necessary, we may write $\xi = (\hat \xi, 0)$ where $\hat \xi \in \RR^m$ for some $m < n$
and $0 \in \RR^{n-m}$.  In fact, we have $\hat \xi \in \hat{ \mathcal Q}$ where
$$
\hat{ \mathcal Q} = \left\{ \hat x \in \RR^m:  \hat x_i > 0 \mbox{ for all } i, \ \sum_{i=1}^m  \hat x_i = 1\right\}.
$$
Note that by assumption that for all $\hat x \in \hat{\mathcal Q}$ we have
$$
b_i(\hat x, 0) = 0 \mbox{ and } \sigma_i( \hat x, 0) = 0 \mbox{ for all } m+1 \le i \le n
$$
 and hence
$$
\sum_{i=1}^m b_i(\hat x, 0) = 0 \mbox{ and } \sum_{i=1}^m \sigma_i( \hat x, 0) = 0.
$$
  By the above argument there
exists a process $(\hat X)_{t \ge 0}$ taking values in $\hat{\mathcal Q}$ such that
$$
d \hat X_t = \hat b(\hat X_t) dt + \hat \sigma(\hat X_t) dW_t
$$
where the functions $\hat b$ and $\hat \sigma$ are defined by $\hat b_i( \hat x) = b_i(\hat x, 0)$
 $\hat \sigma_i( \hat x) = \sigma_i(\hat x, 0)$ for $\hat x \in \hat{\mathcal Q}$ and $1 \le i \le m$.
In particular the process $X = (\hat X, 0)$ solves the original SDE and takes values in $\mathcal Q$
as desired.
\end{proof}

For what follows,  fix points $r_1, \ldots, r_N \in I$.

\begin{lemma}\label{th:density}
For all   $x_1 \ldots x_N \ge 0$ such that $\sum_j x_j =  1$
there exist functions  $\sigma_{i }: \RR^N \to G$ such that
$$
\sigma  \left( \sum_i x_i \delta_{r_i} \right) =   \sum_i \sigma_{i}(x_1, \ldots, x_N) \delta_{r_i}.
$$
\end{lemma}

\begin{proof}
By assumption, since $r_1, \ldots, r_N$ are fixed, there are function $g_i$ such that
$$
\sigma  \left( \sum_i x_i \delta_{r_i} \right) =   \sum_i g_i(x_1, \ldots, x_N) x_i \delta_{r_i}.
$$
Let $\sigma_i(x) = g_i(x) x_i$.
\end{proof}

\begin{proposition}\label{th:Lip}
Let $\sigma: \mathcal P \to  \HS(G,H^*)$ be Lipschitz, so that there
  exists a $K  > 0$  such that
$$
\| \sigma(\mu) - \sigma(\nu) \|_{\HS} \le K  \| \mu - \nu \|_{H^*}.
$$
for all
$\mu, \nu \in \mathcal P$.   Define $\sigma_i: \RR^N \to G$ as in Lemma \ref{th:density}.
Then the functions $\sigma_i$ are Lipschitz.
\end{proposition}

We need a lemma, which amounts to the well-known  fact that
norms on $\RR^N$ are Lipschitz equivalent:

\begin{lemma}\label{th:Lip-finite}
For any $z_1, \ldots, z_N \in G$, there exists constants
$0 < c < C$ such that
$$
c \sum_i \|z_i \|_G \le \left\| \sum_i z_i \delta_{r_i} \right\|_{\HS} \le C \sum_i \|z_i\|_G.
$$
\end{lemma}

\begin{proof}[Proof of Lemma \ref{th:Lip-finite}]
First note by the triangle inequality
\begin{align*}
\left\| \sum_i z_i \delta_{r_i} \right\|_{\HS} & \le  \sum_i \|z_i\|_G \| \delta_{ r_i } \|_{H^*} \\
& \le  C  \sum_i \|z_i\|_G
\end{align*}
where $C = \max_i  \| \delta_{ r_i } \|_{H^*} $.

Now by the inequality
$$
\left | \int \phi(r) \mu(dr) \right| \le \| \mu\|_{H^*} \| \phi \|_{H}
$$
we have
\begin{align*}
\left\| \sum_i z_i \delta_{r_i} \right\|_{\HS}
& \ge \frac{1}{ \| \phi \|_{H} }   \left| \sum_i z_i \phi(r_i) \right|_G
\end{align*}
Let  $\phi_i \in H$ be such $\phi_i(r_i) > 0$ but $\phi_i(r_j) = 0$ for $j \ne i$.
For instance, suppose $r_1 < \ldots < r_N$ and let the graph of $\phi_i$
be a little triangle with base between $r_{i-1}$ and $r_{i+1}$ and vertex at $r_i$
for $1 < i < N$, and the construction appropriately modified for $i=1, N$.  By the above
inequality we have
\begin{align*}
\left\| \sum_i z_i \delta_{r_i} \right\|_{\HS} & \ge \frac{ |\phi_j(r_j) |}{ \| \phi_j\|_{H} }  \|z_j\|_G \mbox{ for all }j \\
& \ge c \sum_i \|z_i\|
\end{align*}
where
$$
c = \frac{1}{N} \min_j \frac{ |\phi_j(r_j) |}{ \| \phi_j\|_{H} }.
$$
\end{proof}

\begin{proof}[Proof of Proposition \ref{th:Lip}]
Let
$$
\mu = \sum_j x_j \delta_{r_j}  \mbox{ and } \nu =  \sum_j y_j \delta_{r_j}.
$$
We have by Lemma \ref{th:Lip-finite} that
\begin{align*}
\sum_i \|\sigma_i(x)-\sigma_i(y)\|_{G} & \le \frac{1}{c} \left\| \sum_i (\sigma_i(x)-\sigma_i(y) )  \delta_{r_i} \right\|_{\HS}  \\
& = \frac{1}{c}  \| \sigma (\mu) - \sigma (\nu) \|_{\HS}    \\
& \le \frac{1}{c} K  \left\| \sum_j (x_j-y_j)  \delta_{r_j} \right\|_{H^*}  \\
 & \le \frac{C}{c} K  \sum_j |x_j-y_j|.
\end{align*}
\end{proof}

Now by Assumption \ref{as:1},
\begin{align*}
&\text{Centering implies} \sum_{i=1}^n \sigma_i(x) = 0 \\
&\text{Lipschitz of } \sigma \text{ implies Lipschitz of } \sigma_i \text{ by Proposition \ref{th:Lip}} \\
&\text{Boundedness in absolute continuity implies } \frac{\|\sigma_i(x)\|}{x_i} \leq C
\end{align*}

And for any $\mu = \sum_{j=1}^n x_j \delta_{r_j}$, we have
$$
(R(\mu) - \rho^* )\mu = x_i \left( \sum_{j=1}^n r_j x_j -r_i \right)
$$
Therefore applying Proposition \ref{th:atomprop}, we have a unique adapted process $(X_t)_{t \geq 0} = (X_t^1,\ldots,X_t^n)_{t \geq 0}$ in $\mathcal{Q}$ and define $\mu_t = \sum_{i=1}^n X_t^i \delta_{r_i}$, we have
\begin{lemma}
$\mu_t$ is a solution to SDE \eqref{eq:evol1} with initial condition $\mu_0 = \sum_{i=1}^n X_0^i \delta_{r_i}$
\end{lemma}
\begin{proof}
Take an arbitrary test function $\phi \in H$, we are going to check that
$$
<\phi, \mu_t> = <\phi, \mu_0> + \int_0^t  <\phi, (R(\mu_s) - \rho^* )\mu_s> ds + \int_0^t <\phi, \sigma(\mu_s)dW_s>
$$
The LHS is clearly $\sum_{i=1}^n X_t^i \phi(r_i)$. For the RHS, first notice that $R(\mu_s) = <r,\mu_s> = \sum_{j=1}^n X_s^j r_s$ and the second term is then
$$
\int_0^t  <\phi, (R(\mu_s) - \rho^* )\mu_s> ds = \sum_{i=1}^n \int_0^t   (\sum_{j=1}^n X_s^j r_s - r_i) X_s^i \phi(r_i)ds
$$
The last term is
\begin{align*}
\int_0^t <\phi, \sigma(\mu_s)dW_s> &= \int_0^t <\phi,\sum_{i=1}^n \sigma_i(X_s)  dW_s \delta_{r_i}> \\
                                   &= \sum_{i=1}^n \int_0^t \sigma_i(X_s)dW_s \phi(r_i)
\end{align*}
Therefore the RHS is given by
$$
\sum_{i=1}^n X_0^i \phi(r_i) + \sum_{i=1}^n \int_0^t   (\sum_{j=1}^n X_s^j r_s - r_i) X_s^i ds \phi(r_i)+ \sum_{i=1}^n \int_0^t \sigma_i(X_s)dW_s \phi(r_i)
$$
Clearly LHS = RHS because $(X_t)_{t \geq 0}$ is a solution in Proposition \ref{th:atomprop}.
\end{proof}
\vspace{3ex}
\noindent
\textit{Step 3: Convergence of atomic solutions}
\begin{lemma}\label{th:festi}
Let $(\mu^n)_{n \geq 0}$ be a sequence of probability measures on $I = [-a,b]$. Given a further $\mu \in H^{*}$ such that $\mu\{-a,b\} = 0$. then
$$
\mu^n \rightarrow \mu \text{ weakly if and only if } \mu^n \rightarrow \mu \text{ in } H^{*}
$$
In particular, $\mu$ is also a probability measure.
\end{lemma}
\begin{proof}
Now suppose $\mu_n \rightarrow \mu$ in $H^{*}$, given any bounded continuous test function $\phi$. Clearly $\phi \in H$ and
\begin{align*}
|<\phi, \mu^n> - <\phi, \mu>| &= |<\phi, \mu^n - \mu>| \\
                              &\leq \|\phi\| \|\mu^n - \mu\| \rightarrow 0
\end{align*}
Hence $\mu^n \rightarrow \mu$ weakly. In particular, taking $\phi(r) = 1$ gives $<1,\mu> = 1$ which proves that $\mu$ is indeed a probability measure.

\vspace{2ex}
For the other direction, assume that $\mu^n \rightarrow \mu$ weakly, then
$$
1 = \mu^n(I) = <1, \mu^n> \rightarrow <1,\mu> = \mu(I) = 1
$$
Fix any $r < 0$, by setting $A = [-a,r)$, we have $\partial A = \{-a,r\}$. Then for almost every $r$ with respect to Lebesgue measure
$$
\mu(\partial A) = \mu\{-a\} + \mu\{r\} = 0
$$
Since there are only countably many discontinuities in distribution function of $\mu$. Hence $\mu(A) - \mu^n(A) \rightarrow 0$ for almost every $r$ by weak convergence. Also $|\mu(A) - \mu^n(A)| \leq 2$ since there are probability measures. Then by dominated convergence theorem:
$$
\int_{-a}^0 \left( \mu(A) - \mu^n(A) \right)^2 dr \rightarrow 0
$$
Similarly,
$$
\int_0^b \left( \mu(r,b] - \mu^n(r,b] \right)^2 dr \rightarrow 0
$$
and hence $\|\mu - \mu^n \| \rightarrow 0$.
\end{proof}

Now we are ready to prove the main theorem \ref{th:solex}
\begin{proof}
Let $\mu_0 \in \mathcal{P}$ be any such probability measure on $I$. Let $(\mu_0^n)_{n \geq 0}$ be a sequence of atomic probability measures that converge weakly to $\mu$. This is always possible since the set of atomic measures is dense. By the previous lemma, we have
$$
\mu_0^n \rightarrow \mu_0 \quad \text{ in } H^{*}
$$
By step 2, we know that for each $n$, there exists a continuous solution $\bar{\mu}^n = (\mu^n_t)_{t \geq 0}$ to the SDE \eqref{eq:evol1} such that $\mu_t^n \in \mathcal{P}$ for each $n,t$. Then by theorem \ref{th:solesti}, for any $T > 0$:
$$
|\| \bar{\mu}^m - \bar{\mu}^n \||_T^2 \leq \| \mu^m_0 - \mu^n_0 \|^2 \exp(KT^2) \rightarrow 0 \text{ as $m,n\rightarrow \infty$}
$$
Hence the sequence $(\bar{\mu}^n)_{n \geq 0}$ is Cauchy in $\mathcal{S}_T$ and therefore tends to some limit $\bar{\mu} \in \mathcal{S}_T$.

In particular fix any $t \leq T$
$$
\EE( \|\mu_t^n - \mu_t \|^2) \leq |\| \mu^n - \mu \||^2 \rightarrow 0
$$
Since $\mu_t^n$ are probability measures, then $1 = \mu^n_t(I) = <1,\mu^n_t>$ and hence:
\begin{align*}
\EE\left( 1 - \mu_t(I) \right)^2 &=  \EE |<1,\mu_t^n> - <1, \mu_t>|^2 \\
                                 &\leq \EE \left( \|1\|_H^2 \|\mu^n_t - \mu_t \|_{H^{*}}^2 \right)  \\
                                 &\leq \EE \left( \|\mu^n_t - \mu_t \|_{H^{*}}^2 \right) \rightarrow 0
\end{align*}
Then $\EE\left( 1 - \mu_t(I) \right)^2 = 0$ implies $\mu_t(I) = 1$ almost surely. Then $\mu_t \in \mathcal{P}$ for all $t$.

Now since $\bar{\mu}^n = (\mu^n_t)_{t \geq 0}$ are solutions, i.e.
$$
\mu_t^n = \mu_0^n + \int_0^t F(\mu_s^n) ds + \int_0^t \sigma(\mu_s^n) dW_s
$$
Now we proceed similarly as in theorem \ref{th:solesti}.
\begin{align*}
|\| \int_0^{\cdot} \sigma(\mu_s^n) dW_s - \int_0^{\cdot} \sigma(\mu_s) dW_s \||^2_T &= \EE \left( \sup_{t \leq T} \| \int_0^t \sigma(\mu_s^n) - \sigma(\mu_s) dW_s \|^2 \right) \\
 \text{(Doob's $L^2$ inequality)} \quad &\leq 4 \EE \int_0^T \| \sigma(\mu_s^n) - \sigma(\mu_s)\|^2 ds \\
 \text{(Lipschitz of $\sigma$)} \quad   &\leq 4 CT |\| \bar{\mu}^n - \bar{\mu}\||_T^2 \rightarrow 0
\end{align*}
Therefore for any $T > 0$
$$
\left(\int_0^t \sigma(\mu_s^n) dW_s \right)_{t \leq T} \rightarrow \left( \int_0^t \sigma(\mu_s) dW_s \right)_{t \leq T} \text{ in } \mathcal{S}_T
$$
Similarly, since $\mu_s^n, \mu_s$ are probability measures for any $s,n$. We could apply lemma \ref{th:festi} and get
$$
\left(\int_0^t F(\mu_s^n) dW_s \right)_{t \leq T} \rightarrow \left( \int_0^t F(\mu_s) dW_s \right)_{t \leq T} \text{ in } \mathcal{S}_T
$$
take $n \rightarrow \infty$ both side we get:
$$
\mu_t = \mu_0 + \int_0^t F(\mu_s) ds + \int_0^t \sigma(\mu_s) dW_s
$$
i.e $\bar{\mu} = (\mu_t)_{t \leq T}$ is a solution. Since $T > 0$ was fixed arbitrarily, we get a solution $\bar{\mu} = (\mu_t)_{t \geq 0}$ to SDE \eqref{eq:evol1} subject to initial condition $\mu_0 \in \mathcal{P}$.
\end{proof}

The generic element of $H^*$ is a distribution but in principle may be much wilder than a signed measure.  However, we
recall this useful fact is Theorem 6.22 of the book of Lieb \& Loss \cite{LL}:
\begin{proposition}
Let
$$
H_+ = \{ \varphi \in H: \varphi(r) \ge 0 \mbox{ for all } r \in I \}.
$$
If $\mu \in H^*$ is such that
$$
\langle \mu, \phi \rangle_{H^*, H} \ge 0 \mbox{ for all } \varphi \in H_+
$$
then $\mu$ is non-negative measure on $I$.
\end{proposition}

\section{Extensions}\label{se:extensions}

Before we launch into a rigorous study of the stochastic evolution equation \eqref{eq:evol}, we briefly
discuss the deterministic version of the equation, where the case where the martingale term is
identically zero.  In this case, we can solve the equation.  Indeed, given a signed measure $\mu_0$
with $\mu_0(\RR) = 1$, let $\mu_t$ be the equivalent signed measure defined by
$$
\mu_t(dr) = \frac{1}{G_t} e^{-rt} \mu_0(dr)
$$
where the normalising constant
$$
G_t = \int_{\RR} e^{-rt} \mu_0(dr)
$$
is assumed positive and finite.
Letting
\begin{align*}
R_t & =  \int_{-\infty}^{\infty}r \ \mu_t(dr) \\
& = - \frac{d}{dt} \log G_t
\end{align*}
we have
\begin{align*}
d\left[\mu_t (dr) \right] & = (R_t-r) \mu_t(dr) dt.
\end{align*}
Note that we also have the identity
\begin{align*}
 e^{ -\int_t^T R_u du} & = \frac{G_T}{G_t} \\
&= \int_{\RR} e^{-(T-t) r} \mu_t(dr)
\end{align*}
as expected.

We  do not restrict our attention to non-negative measures,
since doing so introduces an unexpected constraint.
Indeed,  note that by integrating formally the evolution equation \eqref{eq:evol} we
have
$$
dR_t = - \int_{\RR} (R_t-r)^2 \mu_t(dr) dt + \int_{\RR} r \ M(dr \times dt).
$$
In particular, we see
 that the process $(R_t)_{t\ge 0}$ is a supermartingale if we assume that $\mu_t$ is
  non-negative  for all $t \ge 0$.  That is to say, in order to allow for mean reversion of the
interest rate under the risk-neutral measure, we are forced to work with signed measures.

Here is a result when $\mu_t$ is supported on a finite number of atoms $r_1, \ldots, r_n$.

\begin{proposition}
Let $G$ be a real separable Hilbert space, on which $W$ is a cylindrical Brownian motion.
Fix $n \ge 1$, and let
$$
\mathcal P = \left\{ x= (x_1, \ldots, x_n): x_i \ge 0 \mbox{ for all } i, \sum_{i=1}^n x_i = 1 \right\}
$$
For $1 \le i \le n$, let $b_i: \mathcal P \to \RR$ and $\sigma_i : \mathcal P \to G$ be Lipschitz functions such that
$$
\sum_{i=1}^n b_i(x) = 0 \mbox{ and } \sum_{i=1}^n \sigma_i(x) = 0
$$
for all $x \in \mathcal P$ and such there exists a constant $C  > 0$ such that
$$
| b_i(x) | + \|\sigma_i(x) \| \le C  x_i
$$
for all $x \in \mathcal P$ where $x=(x_1, \ldots, x_n)$.

Then for every $\xi \in \mathcal P$ there exists
unique adapted process $(X_t)_{t \ge 0}$ taking values
in $\mathcal P$ such that $X_0= \xi$ and
$$
dX_t = b(X_t) dt + \sigma(X_t) dW_t.
$$
Furthermore,
$$
X_t^i > 0 \mbox{ for all } t \ge 0 \mbox{ almost surely if and only if } \xi_i > 0.
$$
\end{proposition}

\begin{remark} We are interested in the following
situation.
  Fix a collection of real numbers
$r_1, \ldots, r_n$ and let
$$
b_i(x)  = x_i \left( \sum_{j=1}^n r_j x_j - r_i \right)
$$
for $x=(x_1, \ldots, x_n) \in \mathcal{P}$.
Then the atomic measure valued process
$$
\mu_t = \sum_{i=1}^n X_t^i \delta_{r_i}
$$
evolves according the evolution equation.  Note that the measures
$\mu_s$ and $\mu_t$ equivalent almost surely for all $0 \le s \le t$.
\end{remark}

\begin{proof}
Uniqueness follows from the Lipschitz assumption in the usual way.  We need only prove existence.

Let $\Pi$ be the projection onto the closed convex set $\mathcal P$.
Note that $\Pi$ is Lipschitz, and hence the functions $b \circ \Pi$ and $\sigma \circ \Pi$
are also Lipschitz.   Given $\xi \in \mathcal P$, let
$(X_t)_{t \ge 0}$ be the unique strong solution to the SDE
$$
dX_t = b \circ \Pi(X_t) dt + \sigma\circ \Pi(X_t) dW_t.
$$
which exists by It\^o's theorem.   Note that by summing over the indices,  we have
$$
d\sum_{i=1}^n X^i_t = 0 \Rightarrow \sum_{i=1}^n X^i_t = 1 \mbox{ for all } t \ge 0.
$$
We need only show that $X_t^i \ge 0$  for all $t \ge 0$ and all $i$.  We would be
done since $\Pi(x)=x$ when $x \in \mathcal P$.

First we suppose that $\xi_i > 0$ for all $1 \le i \le n$.
Let $T = \inf\{ t \ge 0: \min \hat X^i_t = 0 \}$.  Note that $X_t \in \mathcal P$ for all $0 \le t \le T$.
We will now show that $T= \infty$ almost surely.

For each $1 \le i \le n$, let
$$
c_i (x) =  \left\{ \begin{array}{cl} \frac{ b_i(x)}{x_i} & \mbox{ if } x_i > 0 \\ 0 & \mbox{ otherwise } \end{array} \right.
$$
and
$$
\tau_i (x) =  \left\{ \begin{array}{cl} \frac{ \sigma_i(x)}{x_i} & \mbox{ if } x_i > 0 \\ 0 & \mbox{ otherwise } \end{array} \right.
$$
where $x=(x_1, \ldots, x_n) \in \mathcal{P}$.
Note that by assumption the functions $c_i$ and $\tau_i$ are bounded and in particular
$$
X_{t \wedge T}^i = \xi_i + \int_0^t X^i_{s \wedge T} dZ_t^i
$$
where $Z^i$ is the continuous semimartingale defined by
$$
Z_t^i = \int_0^t \one_{\{ s \le T\}} [ c_i(X_s) ds + \tau_i(X_s) dW_s].
$$
Hence we can write
$$
X_{t \wedge T}^i = \xi_i e^{ Z_t - [Z]_t/2}.
$$
Since the right-hand side is strictly positive almost surely for all finite $t$, the
event $\{T < \infty\}$ must have probability zero, as claimed.

Now consider the case where there is at least one $i$ such that $\xi_i = 0$.  By relabelling
if necessary, we may write $\xi = (\hat \xi, 0)$ where $\hat \xi \in \RR^m$ for some $m < n$
and $0 \in \RR^{n-m}$.  In fact, we have $\hat \xi \in \hat{ \mathcal P}$ where
$$
\hat{ \mathcal P} = \left\{ \hat x \in \RR^m:  \hat x_i > 0 \mbox{ for all } i, \ \sum_{i=1}^m  \hat x_i = 1\right\}.
$$
Note that by assumption that for all $\hat x \in \hat{\mathcal P}$ we have
$$
b_i(\hat x, 0) = 0 \mbox{ and } \sigma_i( \hat x, 0) = 0 \mbox{ for all } m+1 \le i \le n
$$
 and hence
$$
\sum_{i=1}^m b_i(\hat x, 0) = 0 \mbox{ and } \sum_{i=1}^m \sigma_i( \hat x, 0) = 0.
$$
  By the above argument there
exists a process $(\hat X)_{t \ge 0}$ taking values in $\hat{\mathcal P}$ such that
$$
d \hat X_t = \hat b(\hat X_t) dt + \hat \sigma(\hat X_t) dW_t
$$
where the functions $\hat b$ and $\hat \sigma$ are defined by $\hat b_i( \hat x) = b_i(\hat x, 0)$
 $\hat \sigma_i( \hat x) = \sigma_i(\hat x, 0)$ for $\hat x \in \hat{\mathcal P}$ and $1 \le i \le m$.
In particular the process $X = (\hat X, 0)$ solves the original SDE and takes values in $\mathcal P$
as desired.
\end{proof}

For completeness, the statement and proof of a standard result from convex analysis is
included below.

\begin{lemma}[Orthogonal projection  onto a convex sets is  1-Lipschitz.]  Let  $\mathcal P$  be a
closed convex subset of $\RR^n$.  Then for every $x \in \RR^n$, there
is a point $\Pi(x) \in \mathcal P$ such that
$$
\| x - \Pi(x) \| \le \| x - p \| \mbox{ for all } p \in \mathcal P.
$$
Furthermore, the inequality
$$
\| \Pi(x) - \Pi(y) \| \le \| x-y \|
$$
holds for all $x,y \in \RR^n$.
\end{lemma}

\begin{proof}
First, we show the existence of the projection.  Fix $x \in \RR^k$ and let
$$
d = \inf_{ p \in \mathcal P } \| x - p \|.
$$
Let
$(\pi_k)_k$ be a sequence in $\mathcal P$ such that
$$
\| x -\pi_k \| \to d.
$$
We will show that $(\pi_k)_k$ is Cauchy.
For fixed $k, h$, note that the $\frac{1}{2}(\pi_k+\pi_h) \in \mathcal P$ by convexity,
and hence
$$
\| x- \tfrac{1}{2}(\pi_k+\pi_h) \| \ge d.
$$
Therefore, by the parallelogram law we have
\begin{align*}
\| \pi_k - \pi_h \|^2 & = 2\| x - \pi_k\|^2 + 2\|x - \pi_h\|^2 - 4 \| x- \tfrac{1}{2}(\pi_k+\pi_h) \|^2 \\
& \le 2\| x - \pi_k\|^2 + 2\|x - \pi_h\|^2 - 4 d^2  \\
& \to 2 d^2 + 2d^2 - 4 d^2 = 0.
\end{align*}
By the completeness, the sequence $(\pi_k)_k$ converges to some point $\Pi \in \RR^n$.
And since $\mathcal P$ is closed by assumption, we will have $\Pi \in \mathcal P$.

Now, fix $x$ and $p \in \mathcal P$ and let $p_{\theta} = (1- \theta) \Pi(x) + \theta p$
for $0 \le \theta \le 1$.
Again by convexity, we have the inclusion $p_{\theta} \in \mathcal P$.  Note by the definition of $\Pi(x)$ we have
\begin{align*}
0 & \le \| x - p_{\theta} \|^2 - \| x- \Pi(x) \|^2 \\
& = 2 \theta \langle x - \Pi(x), \Pi(x) - p \rangle + \theta^2 \| \Pi(x) - p \|^2.
\end{align*}
Sending $\theta \downarrow 0$ in the above inequality yields the conclusion that
$$
 \langle  \Pi(x) - x , \Pi(x) - p \rangle  \le 0
$$
for all $p \in \mathcal P$.  Hence
\begin{align*}
\| \Pi(x) - \Pi(y) \|^2 & = \langle \Pi(x) - x, \Pi(x)- \Pi(y) \rangle + \langle y - \Pi(y), \Pi(x)- \Pi(y) \rangle
+ \langle x-y, \Pi(x)- \Pi(y) \rangle \\
& \le \langle x-y, \Pi(x)- \Pi(y) \rangle \\
& \le \| x-y\| \| \Pi(x) - \Pi(y) \|
\end{align*}
from which the Lipschitz bound follows.
\end{proof}

\end{document}